\documentclass[12pt]{amsart}
\usepackage[mathscr]{eucal}
\usepackage{amssymb}
\usepackage{amsfonts}

\theoremstyle{plain} 
\newtheorem{thm}{Theorem}[section]
\newtheorem{cor}[thm]{Corollary}
\newtheorem{lem}[thm]{Lemma}
\newtheorem{prop}[thm]{Proposition}

\theoremstyle{definition}

\newtheorem{ex}[thm]{Example}

\theoremstyle{remark}
\newtheorem{rem}[thm]{Remark}

\def\fp{{\mathfrak p}}
\def\fm{{\mathfrak m}}

\numberwithin{equation}{section}

\def\NN{{\mathbb N}}

\def\ba{{\mathbf a}}

\def\Sym{\operatorname{Sym}}

\def\height{\operatorname{ht}}

\def\Spec{\operatorname{Spec}}
\def\rank{\operatorname{rank}}

\def\chara{\operatorname{char}}

\def\<{{\langle}}
\def\>{{\rangle}}

\def\ISp{I^{\rm Sp}}
\def\IVd{I^{\rm Vd}}

\begin{document}
\title{Vandermonde determinantal ideals}
\author{Junzo Watanabe}
\address{Department of Mathematics, Tokai University, Hiratsuka,, 259-1292 Japan}
\email{
watanabe.junzo@tokai-u.jp}
\author{Kohji Yanagawa}
\address{Department of Mathematics, Kansai University, Suita, Osaka 564-8680, Japan}
\email{yanagawa@kansai-u.ac.jp}
\thanks{The second authors is partially supported by JSPS Grant-in-Aid for Scientific Research (C) 16K05114.}
\maketitle

\begin{abstract}
We  show that the ideal generated by maximal minors (i.e., $k+1$-minors) of a  $(k+1) \times n$ Vandermonde matrix is radical and Cohen-Macaulay. Note that this ideal is generated by all Specht polynomials with  shape $(n-k,1, \ldots,1)$. 
\end{abstract}

\section{Introduction}
Let $n,k$ be integers with $n > k \ge 1$. Consider the polynomial ring $R=K[x_1, \ldots, x_n]$ over a field $K$, and the following non-square Vandermonde matrix
$$M_{n,k}:=\left( \begin{array}{llll}1&1&\cdots&1\\x_1&x_2&\cdots &x_n\\ x_1^2&x_2^2&\cdots &x_n^2\\\vdots&\vdots&\cdots &\vdots\\
x_1^k&x_2^k&\cdots &x_n^k\end{array}\right).$$
Let  $\IVd_{n,k}$ denote the ideal of $R$ generated by all maximal minors (i.e., $k+1$ minors) of $M_{n,k}$. 

The purpose of this paper is to prove the following.

\begin{thm}
$R/\IVd_{n,k}$ is a reduced Cohen-Macaulay ring with $\dim R/\IVd_{n,k}=k$ and  $\deg R/\IVd_{n,k}=S(n,k)$, where $S(n,k)$ stands for the Stirling number of the second kind. 
\end{thm}

The present paper can be seen as the  precursor of our ongoing project \cite{WY} on {\it Specht ideals}. 
For a partition $\lambda$ of $n$,  we can consider the ideal 
$$\ISp_\lambda= (\Delta_T \mid \text{$T$ is a Young tableau of shape $\lambda$}) $$ 
of $R$, where $\Delta_T \in R$ denotes the {\it Specht polynomial} corresponding to $T$ (see \cite{LefPro}). 
Then we have $\IVd_{n,k} =\ISp_{(n-k,1,\ldots, 1)}$. 
Note that the $K$-vector subspace $\<\,  \Delta_T \mid \text{$T$ is a Young tableau of shape $\lambda$} \, \>$ of $R$ is the {\it Specht module} associated with $\lambda$ as an $S_n$-module. 
The Specht modules are often constructed in different manner (e.g., using Young tabloids), and play crucial role in the representation theory of symmetric groups (see, for example \cite{Sa}).   
General Specht ideals are much more delicate than the Vandermont case $\IVd_{n,k}$. 
For example,  $R/\ISp_\lambda$ is not even pure dimensional. for many $\lambda$, and 
the Cohen-Macaulayness of $R/\ISp_\lambda$ may depend on $\chara(K)$ for some fixed  $\lambda$.   
In \cite{WY}, we will use the representation theory of symmetric groups. 

It is noteworthy that Fr\"{o}berg and Shapiro  \cite{FS} also studied  some variants of $R/\IVd_\lambda$ 
in a different context. 

\section{Results and proofs}
Extending the base field, we may assume that $K$ is algebraically closed. Theoretically, this assumption is not necessary in the following argument, but it makes the expositions more readable.

For an ideal $I \subset R$, set $V(I):=\{ \fp \mid \fp \in \Spec R,  \fp \supset I\}$ as usual. 
For $\ba=(a_1, \ldots, a_n) \in K^n$, let $\fm_\ba$ denote the maximal ideal 
$(x_1 -a_1, x_2 - a_2, \ldots, x_n - a_n)$ of $R$.  
By abuse of notation, we just write $\ba \in V(I)$ to mean $\fm_\ba \in V(I)$. 
Clearly, $\ba \in V(I)$ if and only if $f(\ba)=0$ for all $f \in I$.

\begin{prop}
We have $\dim R/\IVd_{n,k}=k$ and
$$\deg \left(R/\sqrt{\IVd_{n,k}} \right)=S(n,k),$$ 
where $S(n,k)$ stands for the Stirling number of the second kind, that is, 
the number of ways to partition the set $\{1,2, \ldots, n\}$  into $k$ non-empty subsets. 
\end{prop}

\begin{proof}
For  $\ba = (a_1, \ldots, a_n) \in K^n$, $\ba \in V(\IVd_{n,k})$ if and only if 
$$\rank (M_{n,k}(\ba)) \le k,$$ 
where  $M_{n,k}(\ba)$ is the matrix given by putting $x_i=a_i$ for each $i$ in $M_{n,k}$. 
The latter condition is equivalent to that $\# \{a_1, \ldots, a_n \} \le k$. 
This is also equivalent to that there is a partition 
$\Pi =\{F_1, \ldots, F_k \}$ of  the set $[n] := \{1,2, \ldots, n\}$ such that $a_i=a_j$ for all $i,j  \in F_l$ ($l=1,2, \ldots, k$). 
For the above partition $\Pi$, let $P_\Pi$ denote the prime ideal 
$$(x_i-x_j \mid \text{$i,j \in F_l$ for $l=1,\ldots, k$})$$
of $R$. Since
$$\sqrt{\IVd_{n,k}}=\bigcap_{\substack{\text{$\Pi$: partition of $[n]$} \\ \# \Pi =k}} P_\Pi,$$
$\dim R/P_\Pi =k$ for all $\Pi$, and $\deg R/P_\Pi =1$, we are done. 
\end{proof}

Applying elementary column operations to $M_{n,k}$, we get the following matrix  
$$M'_{n,k}:=\left( \begin{array}{ccccc}1&0&0&\cdots&0\\x_1&x_2-x_1&x_3-x_1&\cdots &x_n-x_1\\ x_1^2&x_2^2-x_1^2&x_3^2-x_1^2&\cdots &x_n^2-x_1^2\\\vdots&\vdots&\vdots &\cdots &\vdots\\
x_1^k&x_2^k-x_1^k&x_3^k-x_1^k&\cdots &x_n^k-x_1^k\end{array}\right).$$ 
Consider its  $k \times (n-1)$ submatrix
 $$N_{n,k}:=\left( \begin{array}{cccc}x_2-x_1&x_3-x_1&\cdots &x_n-x_1\\ x_2^2-x_1^2&x_3^2-x_1^2&\cdots &x_n^2-x_1^2\\\vdots&\vdots &\cdots &\vdots\\
x_2^k-x_1^k&x_3^k-x_1^k&\cdots &x_n^k-x_1^k\end{array}\right).$$
Clearly, $\IVd_{n,k}$ is generated by all maximal minors (i.e., $k$-minors) of $N_{n,k}$. 

\begin{thm}\label{Vandermonde main}
$R/\IVd_{n,k}$  is Cohen-Macaulay. Moreover, its minimal graded free resolution is given by the Eagon-Northcott complex (see, for example \cite{M08}) associated to the matrix  $N_{n,k}$.
\end{thm}

\begin{proof}
Since $\height(\IVd_{n,k}) = \dim R -\dim R/\IVd_{n,k} = n-k = (n-1)-k+1$, $\IVd_{n,k}$ is a {\it standard determinantal ideal} in the sense of \cite{M08}. Hence the assertion is immediate from well-known properties of this notion (c.f.  \S1.2 of \cite{M08}). 
\end{proof}

When we construct the Eagon-Northcott resolution of $\IVd_{n,k}$, we use  a symmetric power $\Sym_i V$ 
of a $k$-dimensional vector space $V$with a basis $e_1, \ldots, e_k$ such that $\deg e_i =i$ for each $i$. 
Set
$$p^{m}_{i,j} :=\# \{ (a_1, \ldots, a_m) \in \NN^m \mid a_1 + a_2 +\cdots + a_m =i, \, a_1 + 2 a_2 +\cdots + m a_m =j \}$$
For simplicity, set $p^m_{0,j} := \delta_{0,j}$. 
The following facts are easy to see.  
\begin{itemize}
\item $p^{m}_{i,j} \ne 0$ if and only if $i \le j \le im$,
\item $\sum_j p^{m}_{i,j} = \binom{m+i-1}{i}$. 
\end{itemize}
For  the vector space $V$ discussed above,  the dimension of the degree $j$ part of $\Sym_i V$ is $p^k_{i,j}$.

\begin{cor}\label{betti numbers}
For $i \ge 1$, we have 
$$\beta_{i,j}(R/\IVd_{n,k}) = p^k_{i-1, j-\frac{k(k+1)}2 } \times \binom{n-1}{k+i-1}.$$ 
\end{cor}

\begin{proof}
Since the minimal free resolution of $R/\IVd_{n,k}$ is given by  the Eagon-Northcott complex, we have  
\begin{eqnarray*}
\beta_{i,j}(R/\IVd_{n,k})  &= &
\left(\dim_K [ \, (\Sym_{i-1} V) \otimes_K \bigwedge^k V \, ]_{j} \right) \times (\dim_K \bigwedge^{k+i-1} W)\\
&= &(\dim_K [ \, \Sym_{i-1} V \, ]_{j-\frac{k(k+1)}2} ) \times (\dim_K \bigwedge^{k+i-1} W)\\
&=& p^k_{i-1, j-\frac{k(k+1)}2 } \times \binom{n-1}{k+i-1},
\end{eqnarray*}
where $V$ is the $K$-vector space considered above, and 
$W$ is a $K$-vector space of dimension $n-1$. 
\end{proof}

\begin{ex}\label{3 tables}
Since $p^2_{i,j} =0$ or $1$ for all $i,j$, we have $\beta_{i,j}(R/\IVd_{n,2}) = 0$ or $\binom{n-1}{i+1}$ for all $i \ge 1$.  For example, the  Betti table of $R/\IVd_{6,2}$ is the following. 
\begin{verbatim}
     total: 1 10 20 15 4
         0: 1  .  .  . .
         1: .  .  .  . .
         2: . 10 10  5 1
         3: .  . 10  5 1
         4: .  .  .  5 1
         5: .  .  .  . 1
\end{verbatim}

The following are the Betti tables of $R/\IVd_{6,3}$ and $R/\IVd_{7,3}$, respectively. 
\begin{verbatim}
      total: 1 10 15 6
          0: 1  .  . .
          1: .  .  . .
          2: .  .  . .
          3: .  .  . .
          4: .  .  . .
          5: . 10  5 1
          6: .  .  5 1
          7: .  .  5 2
          8: .  .  . 1
          9: .  .  . 1
\end{verbatim}

\medskip 

\begin{verbatim}
     total: 1 20 45 36 10
         0: 1  .  .  .  .
         1: .  .  .  .  .
         2: .  .  .  .  .
         3: .  .  .  .  .
         4: .  .  .  .  .
         5: . 20 15  6  1
         6: .  . 15  6  1
         7: .  . 15 12  2
         8: .  .  .  6  2
         9: .  .  .  6  2
        10: .  .  .  .  1
        11: .  .  .  .  1
\end{verbatim}
\end{ex}


\begin{thm}\label{degree}
We have $$\deg R/\IVd_{n,k} = S(n,k).$$
\end{thm}

\begin{proof}
Since $\IVd_{n,1}$ is an ideal generated by linear forms, we have 
$\deg R/\IVd_{n,1} = 1 = S(n,1)$ for all $n \ge 2$. Since $\IVd_{n,n-1}$ is a principal ideal generated by a polynomial of degree $\binom{n}{2}$, we have  $\deg R/\IVd_{n,n-1} = \binom{n}{2} = S(n,n-1)$. 
It is well-known that the Stirling numbers of the second kind satisfy the recurrence relation
$$S(n, k) = S(n-1,k-1)+kS(n-1,k).$$
So it suffices to show that $\deg R/\IVd_{n,k}$ also satisfies the corresponding relation 
\begin{equation}\label{recursion}
\deg R/\IVd_{n,k}  = \deg  R'/\IVd_{n-1,k-1}  + k(\deg  R'/\IVd_{n-1,k} )
\end{equation}
for $n-1 > k$, where $R'$ is the polynomial ring $K[x_1, \ldots, x_{n-1}]$. 

From now on, we assume that $n-1 > k$. 
Note that the matrices 
$$N_{n-1,k-1}:=\left( \begin{array}{cccc}x_2-x_1&x_3-x_1&\cdots &x_{n-1}-x_1\\ x_2^2-x_1^2&x_3^2-x_1^2&\cdots &x_{n-1}^2-x_1^2\\\vdots&\vdots &\cdots &\vdots\\
x_2^{k-1}-x_1^{k-1}&x_3^{k-1}-x_1^{k-1}&\cdots &x_{n-1}^{k-1}-x_1^{k-1}\end{array}\right)$$
and 
$$N_{n-1,k}:=\left( \begin{array}{cccc}x_2-x_1&x_3-x_1&\cdots &x_{n-1}-x_1\\ x_2^2-x_1^2&x_3^2-x_1^2&\cdots &x_{n-1}^2-x_1^2\\\vdots&\vdots &\cdots &\vdots\\
x_2^k-x_1^k&x_3^k-x_1^k&\cdots &x_{n-1}^k-x_1^k\end{array}\right)$$
can be regarded as submatices of $N_{n.k}$. Let $J_1$ and $J_2$ be the ideals (of $R$) generated by all maximal minors of $N_{n-1,k-1}$ and $N_{n-1,k}$, respectively. By \cite[Lemma~2.3 (2)]{M06}, we have 
$$\deg R/\IVd_{n,k} = \deg R/J_1 + k(\deg R/J_2).$$
On the other hand, we have $R/J_1 \cong ( R'/\IVd_{n-1,k-1})[x_n]$, and hence $\deg R/J_1 = \deg  R'/\IVd_{n-1,k-1}$. Similarly,  $\deg R/J_2 = \deg  R'/\IVd_{n-1,k}$. 
Now \eqref{recursion} is clear. 
\end{proof}

\begin{rem}
In the first version of this paper, the key formula \eqref{recursion} was shown by a direct computation from  Corollary~\ref{betti numbers}. More precisely, the equations
 $$\beta_{1,j}(R/\IVd_{n,k})=\beta_{1, j-k}(R/\IVd_{n-1,k-1})+\beta_{1,j}(R/\IVd_{n-1,k})$$
and 
$$\beta_{i,j}(R/\IVd_{n,k})=\beta_{i, j-k}(R/\IVd_{n-1,k-1})+\beta_{i,j}(R/\IVd_{n-1,k})+\beta_{i-1,j-k}(R/\IVd_{n-1,k})$$
hold for $i \ge 2$. One can check this in Example~\ref{3 tables}.  Anyway, we see that these equations imply 
\eqref{recursion}. 

Now we know that \eqref{recursion} is a direct consequence of \cite[Lemma~2.3 (2)]{M06}. It is noteworthy that this lemma is a result of Gorenstein liaison theory. 
\end{rem}

\begin{cor}
$R/\IVd_{n,k}$ is reduced. 
\end{cor}

\begin{proof}
Since $A:=R/\IVd_{n,k}$ is Cohen--Macaulay, any non-zero ideal $I \subset A$ 
satisfies $\dim I =\dim A$ as an $A$-module. Hence if $A$ is not reduced, then $\deg A > \deg A/\sqrt{(0)}$. However, it contradicts the fact that 
$$\deg \left(R/\IVd_{n,k} \right)=S(n,k)=\deg \left(R/\sqrt{\IVd_{n,k}} \right).$$
\end{proof}

\section*{ Acknowledgements}
We are grateful to the anonymous referee for telling us \cite[Lemma~2.3]{M06}, which drastically simplified the proof of Theorem~\ref{degree}.


\begin{thebibliography}{8}


\bibitem{FS} R. Fr\"{o}berg, and B. Shapiro, On Vandermonde varieties, Math. Scand.  {\bf 119}  (2016),  73--91. 


\bibitem{LefPro} T. Harima, T. Maeno, H. Morita, Y. Numata, A. Wachi, and J. Watanabe, The
Lefschetz Properties, Springer Lecture Notes 2080, Springer-Verlag, 2013.

\bibitem{M06} R. M. Mir\'o-Roig, A note on the multiplicity of determinantal ideals, J. Algebra {\bf 299} (2006), 714--724.

\bibitem{M08} R.M. Mir\'o-Roig, Determinantal Ideals, Progress in Mathematics, vol. 264, Birkh\"{a}user Verlag, Basel, 2008.


\bibitem{Sa} B.E. Sagan, The Symmetric Group: Representations, Combinatorial Algorithms, and Symmetric Functions, second edition, Springer-Verlag,  2001.


\bibitem{WY} J. Watanabe and K. Yanagawa, On Specht ideals, in preparation. 
\end{thebibliography}
\end{document}